\documentclass[11pt, reqno]{amsart}
\usepackage[margin=3.6cm]{geometry}
\usepackage{array,booktabs,tabularx}
\usepackage{graphicx}
\usepackage{amssymb}
\usepackage{enumerate}
\usepackage{color}
\usepackage{hyperref}
\usepackage{esint}
\usepackage{mathtools}
\usepackage{dsfont}
\usepackage{ esint }
\usepackage[show]{ed}
\usepackage{comment}

\newtheorem{theorem}{Theorem}

\newtheorem{lemma}[theorem]{Lemma}

\theoremstyle{definition}
\newtheorem{definition}{Definition}
\newtheorem{remark}{Remark}

\newcommand{\R}{{\mathbb R}}
\newcommand{\Z}{{\mathbb Z}}

\newcommand{\U}{{\mathcal U}}
\newcommand{\G}{{\mathcal G}}

\newcommand{\ep}{\varepsilon}

\newcommand{\Id}{\operatorname{Id}}

\def\XXint#1#2#3{{\setbox0=\hbox{$#1{#2#3}{\int}$} \vcenter{\hbox{$#2#3$}}\kern-.5\wd0}}

\DeclareMathOperator{\supp}{supp}

\newcommand{\e}{\varepsilon}

\title[Eigenfunction {averages}]{Averages of Eigenfunctions Over Hypersurfaces}

\author{Yaiza Canzani}
\address{Department of Mathematics, University of North Carolina, Chapel Hill, NC, USA}
\email{canzani@email.unc.edu}
\author{Jeffrey Galkowski}
\address{Department of Mathematics, Stanford University, Stanford, CA, USA}
\email{jeffrey.galkowski@stanford.edu }
\author{John A. Toth}
\address{Department of Mathematics and Statistics, McGill University, Montr\'eal, QC, Canada}
\email{jtoth@math.mcgill.ca}

\begin{document}
\maketitle
\begin{abstract} Let $(M,g)$ be a compact, smooth,  Riemannian manifold and $\{ \phi_h \}$  an $L^2$-normalized sequence of Laplace eigenfunctions with defect measure $\mu$. Let $H$ be a smooth hypersurface with unit exterior normal $\nu.$
Our main result  says  that  when $\mu$ is {\em not} concentrated conormally to $H$, the eigenfunction restrictions to $H$ satisfy
$$ \int_H \phi_h d\sigma_H = o(1) \qquad  \text{and}\qquad  \int_H h D_{\nu} \phi_h d\sigma_H = o(1),$$
$h \to 0^+$.
\end{abstract}

\section{Introduction}

On a compact Riemannian manifold $(M,g)$, with no boundary, consider a sequence of Laplace eigenfunctions $\{\phi_h\}$, 
\[-h^2 \Delta_g \phi_h=\phi_h,\]
normalized so that $\|\phi_h\|_{L^2(M)}=1$. The goal of this article is to study the average oscillatory behavior of $\phi_h$ when restricted to a hypersurface $H \subset M$.  Namely, the goal is to find a condition on the pair $(\{\phi_h\}, H)$ so that
\begin{equation}\label{E:goal}
\int_H \phi_h \, d\sigma_{{H}}=o(1),
\end{equation}
as $h \to 0^+$, where $\sigma_H$ denotes the hypersurface measure on $H$ induced by the Riemannian structure. 

It is important to point out that one cannot always expect to observe this oscillatory decay. For instance, on the round sphere, zonal harmonics of even degree integrate to a constant along the equator. Also, for any closed geodesic inside the square flat torus. there is a sequence of eigenfunctions that  integrate to a non-zero constant.

Integrals of the form ~\eqref{E:goal} have been studied for quite some time, going back to the work of Good~\cite{Good} and Hejhal~\cite{Hej} that treated the case where $H$ is a periodic geodesic inside a compact hyperbolic manifold. These authors proved that in such a case,  $\int_H \phi_h \, d\sigma_{{H}}=O(1)$ as $h\to 0^+$. Zelditch~\cite{Zel} {generalized this} to the case where $H$ is any hypersurface inside a compact manifold, {showing that for any hypersurface $H$,
\begin{equation}
\label{e:stdBound}
\int_H\phi_hd\sigma_H=O(1).
\end{equation} }
{In addition,} it follows from \cite{Zel} that for a density one subsequence of eigenvalues $\{h_j\}_j,$ one has $\lim_{j\to \infty}\int_H \phi_{h_j} \, d\sigma_H=0.$
Moreover, one can actually get an explicit polynomial bound of the form $O(h^{1/2 - 0})$ for the rate of decay of expectations for the density-one subsequence (see \cite{JZ}). However, the latter estimate is not satisfied for all eigenfunctions and it is not clear which sequence of  eigenfunctions must be removed for the estimate to hold.  There are several articles that address this issue by restricting to special cases of  Riemannian surfaces $(M,g)$ and  special curves $H \subset M.$  Working on surfaces of strictly negative curvature, and choosing $H$ to be a geodesic,  Chen-Sogge \cite{CS} {proved $\int_H \phi_h \, d\sigma_{{H}} =o(1)$}.  Subsequently, Sogge-Xi-Zhang \cite{SXZ} obtained a $O((\log h)^{-1/2})$ bound on the rate of decay under a relaxed curvature condition. Recently, working on surfaces of non-positive curvature Wyman \cite{Wym} obtained \eqref{E:goal} when assuming curvature conditions on $H$.
Finally, we remark that on average, one expects $\int_H \phi_h \, d\sigma_{{H}} \asymp h^{\frac{1}{2}}$ (see \cite{Esw}).

In this article we focus on establishing~\eqref{E:goal} given explicit conditions on the sequence of eigenfunctions $\{\phi_h\}$. We do not impose any geometric conditions on $(M,g)$, nor do we assume it is a surface. Furthermore, we do not restrict our attention to geodesic curves and allow $H$ to be any hypersurface in $M$.  Instead, we prove that~\eqref{E:goal} holds provided that the sequence $\{\phi_h\}$ does not asymptotically concentrate in the conormal direction $N^*H$ to $H$. One example where this holds is the case quantum ergodic sequences of eigenfunctions and any hypersurface $H$.

\subsection{Statements of the results} 
Let $H\subset M$ be a closed smooth hypersurface, and write $S^*_HM \subset S^*M$ for the space of unit covectors with foot-points in $H$, and $S^*H$ for the set of unit covectors tangent to $H$. \
We fix $t_0>0$ small enough and define a measure $\mu_H$ on $S^*_HM \subset S^*M$  by 
\begin{equation} \label{e:muH2}
\mu_H(A) :=\frac{1}{2t_0}\, \mu\Big( \bigcup_{|s|\leq t_0}G^s(A)\Big),
\end{equation}
where  $G^t:S^*M\to S^*M$ denotes the geodesic flow.
Remark ~\ref{l:pushforward} shows that if $A \subset S^*_HM$ is so that  $\overline{A} \subset S_H^*M \backslash S^*H$, then $\mu_H(A)$ is independent of the choice of $t_0$ and it is natural to replace fixed $t_0$ with $\lim_{t_0\to 0}$.
\begin{definition} \label{defn}
We say that $\mu$ is \emph{conormally diffuse with respect to $H$} if \[\mu_H(N^*H)=0.\]
If  $U\subset H$ is open, we say that $\mu$ is \emph{conormally diffuse with respect to $H$ over $U$} if 
\[\mu_H(N^*H\cap S^*_UM)=0.\]
\end{definition}

 As an example, this condition is satisfied when $\{\phi_h\}$ is a quantum ergodic (QE) sequence and $\mu = \mu_L$, the Liouville measure on $S^*M.$
Note that the QE condition is much stronger than the assumption in Definition \ref{defn}.
In Section \ref{S: examples} we give examples of hypersurfaces and sequences of eigenfunctions for which the defect measure is conormally diffuse but is not absolutely continuous with respect to the Liouville measure.
Our main result is the following.
\begin{theorem}\label{T: main}
Let $H \subset M$ be a closed hypersurface. Let  $\{\phi_h\}$ be a sequence of eigenfunctions associated to a defect measure $\mu$ that is conormally diffuse with respect to $H$. Then,
\[\int_H \phi_{h}\, d{\sigma}_H=o(1),\]
and 
\[\int_H  h \partial_\nu \phi_{h}\, d{\sigma}_H=o(1),\]
as $h  \to 0^+$.
\end{theorem}

{\begin{remark} The proof of Theorem \ref{T: main} actually shows that $\int_{H} \phi_h \chi d\sigma_H = o(1)$ for any $\chi \in C^{\infty}(H).$ 
We note also that the methods of this paper give another independent proof of~\eqref{e:stdBound}.
\end{remark}
}

As we have already pointed out, the Liouville measure $\mu = \mu_L$ is conormally diffuse. Consequently, the following result is a corollary of Theorem \ref{T: main}:
\begin{theorem} \label{cor1}
Le $H \subset M$ be a closed hypersurface and $\{ \phi_h \}$ be any 
QE sequence sequence of eigenfunctions. Then,
\[\int_H \phi_{h}\, d{\sigma}_H=o(1) \qquad \text{and} \qquad  \int_H h \partial_\nu \phi_{h}\, d{\sigma}_H=o(1).\]

\end{theorem}
By Lindenstrauss' celebrated result \cite{L}, Hecke eigenfunctions on compact, arithmetic hyperbolic surfaces are all QE (ie. they are quantum uniquely ergodic (QUE)). Together with Theorem \ref{cor1} this yields
\begin{theorem} \label{cor2}
Let $(H/\Gamma,g)$ be a compact, arithmetic surface and  $H \subset M$ be a closed, $C^{\infty}$ curve. Then, for all Hecke eigenfunctions $\{\phi_h\},$
\[\int_H \phi_{h}\, d{\sigma}_H=o(1)  \qquad \text{and} \qquad   \int_H h \partial_\nu \phi_{h}\, d{\sigma}_H=o(1).\]
\end{theorem}

One can localize the results in Theorems \ref{T: main}-\ref{cor2}. In the following, we write $ d\sigma_H $ for the measure on $H$ induced by the Riemannian structure.

\begin{theorem} \label{local}
Let $(M,g)$ be a smooth, closed Riemannian manifold and $H \subset M$ be a closed hypersurface with $A \subset H$ a subset with piecewise $C^{\infty}$ boundary and suppose $U\subset H$ is open with $\overline{A}\subset U$. Let $\{\phi_h\}$ be  a sequence of eigenfunctions with defect measure $\mu$ conormally diffuse with respect to $H$ over $U$. Then, \[\int_A \phi_{h}\, d{\sigma}_H=o(1),\]
and 
\[\int_A  h \partial_\nu \phi_{h}\, d{\sigma}_H=o(1),\]
as $h  \to 0^+$.
\end{theorem}

\begin{remark} We note that as a corollary of Theorem \ref{local}, the results in Theorems \ref{cor1} and \ref{cor2} for QE eigenfunctions extend to all smooth curve segments $A.$ \end{remark}

\noindent {\sc Acknowledgements.} The authors would like to thank the anonymous referee for many helpful comments. J.G. is grateful to the National Science Foundation for support under the Mathematical Sciences Postdoctoral Research Fellowship  DMS-1502661.  The research of J.T. was partially supported by NSERC Discovery Grant \# OGP0170280 and an FRQNT Team Grant. J.T. was also supported by the French National Research Agency project Gerasic-ANR-
13-BS01-0007-0.


\section{Decomposition of defect measures} \label{cross-section}
\subsection{Invariant Measures near transverse submanifolds}
Let $N$ be a smooth manifold, $\mathcal{V}$ be a vector field on $N$ and write $\varphi^{\mathcal{V}}_t:N \to N$ for the flow map generated by $\mathcal V$ at time $t$.
Let $\Sigma\subset N$ be a smooth manifold transverse to $\mathcal{V}$. Then for $\epsilon>0$ small enough, the map $\iota:(-2\epsilon,2\epsilon)\times \Sigma \to N$
\[\iota(t,q)=\varphi^{\mathcal{V}}_t(q)\]
is a diffeomorphism onto its image and we may use $(-2\epsilon,2\epsilon)\times \Sigma$ as coordinates on $N$ near $\Sigma$. 
\begin{lemma}\label{l:invariance}
Suppose that $\mu$ is a finite Borel measure on $N$ and that $\mathcal{V}\mu=0$ i.e. $(\varphi_t^{\mathcal{V}})_*\mu=\mu$. Then, for a Borel set $A\subset [-\e,\e)\times \Sigma$,
$$
\iota^*\mu(A)=dt d\mu_\Sigma (A)
$$
where $d\mu _\Sigma$ is a finite Borel measure on $\Sigma$.
\label{l:invMeasure}
\end{lemma}
\begin{proof}
{As above, we choose coordinates $(t,q)$ so that $\iota^*\mathcal{V}=\partial_t$. Then, for all $F\in C_c^\infty (-2\epsilon,2\epsilon)\times \Sigma$, 
$$\int \partial_t Fd\mu=0.$$
Now, fix $\chi\in C_c^\infty((-2\epsilon,2\e))$ with with $\int\chi dt=1$. Let $f\in C_c^\infty((-2\epsilon,2\e)\times \Sigma)$ and define 
$$\bar{f}(q):=\int f(t,q)dt.$$
Then $f(t,q)-\chi(t)\bar{f}(q)=\partial_t F$ with 
$$F(t,q):=\int_{-\infty}^t f(s,q)-\chi(s)\bar{f}(q)ds\in C_c^\infty((-2\e,2\e)\times \Sigma).$$
Therefore, for all $f\in C_c^\infty((-2\e,2\e)\times N) $ and $\chi \in C_c^\infty((-2\e,2\e))$ with $\int \chi dt=1$,
$$\int f(t,q)d\mu(t,q)=\int \chi(t)\bar{f}(q)d\mu(t,q) =\iiint f(s,q)ds\chi(t)d\mu(t,q).$$
Now, let $B\subset \Sigma$ be Borel and $I\subset (-2\e,\e)$ Borel and $f_n(t,q)\uparrow 1_{I}(t)1_{B}(q)$. Then by the dominated convergence theorem,
$$\mu(I\times B)=\iint |I|1_{B}(q)\chi(t)d\mu(t,q).$$
Next, let $\chi_n\uparrow (2\e)^{-1}1_{[-\e,\e]}$ with $\int \chi_n \equiv 1$. Then we obtain
$$\mu(I\times B)=\frac{|I|}{2\e}\mu([-\e,\e]\times B).$$
So, letting $\mu_{\Sigma}(B):=(2\e)^{-1}\mu([-\e,\e]\times B)$, we have that for rectangles $I\times B$, $\mu(I\times B)=dtd\mu_\Sigma(I\times B)$. But then, since these sets generate the Borel sigma algebra, the proof of the lemma is complete.
}

\end{proof}

\subsection{Fermi coordinates}
{Throughout the remainder of the article we will work in the case that $H\subset M$ is a smooth, orientable, separating hypersurface. That is,  $M\setminus H$ has two connected components. We then recover Theorem~\ref{T: main} for general $H$ after proving Theorem~\ref{local} for such hypersurfaces. We then divide a given hypersurface into finitely many (possibly overlapping) subsets of separating orientable hypersurfaces and apply Theorem~\ref{local} to each.}
Let $H \subset M$ be a closed smooth hypersurface  and let $U_H$ be a Fermi collar neighborhood of $H$.  In Fermi coordinates \[ U_H  = \{(x',x_n): \, x' \in H \; \text{and}\;x_n \in (-c, c)\}\]
for some $c>0$,  and $H=\{(x',0) : x' \in H\}$. Since $H$ is a closed, {{separating}} hypersurface, it divides $M$ into two connected components $\Omega_H$ and $M \backslash \Omega_H$.  In the Fermi coordinates system, the point $(x',x_n)$ is identified with the point $\exp_{x'}(x_n \nu_n) \in U_H$ where $\nu_n$ is the unit normal vector to $\Omega_H$ with base point at  $x' \in H$. 

\begin{center}
\includegraphics[width=9cm]{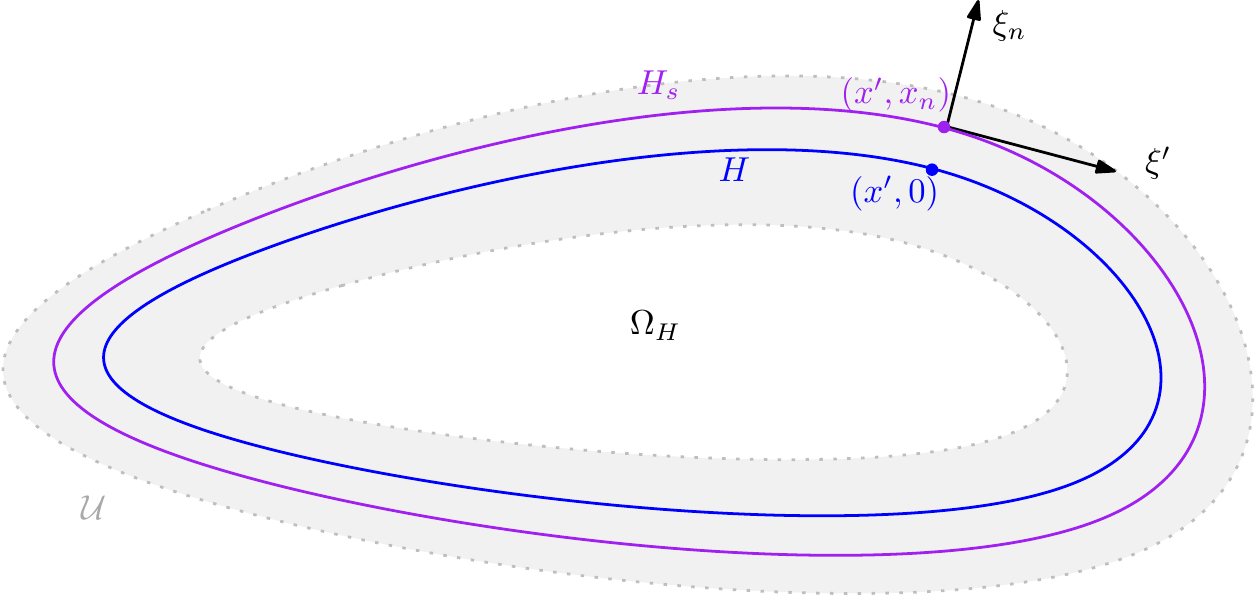}
\end{center}

The Fermi coordinates on $U_H$ induce coordinates $(x',x_n, \xi', {\xi_n})$ on $S_\U^*M=\{(x,\xi)\in S^*M: \; x \in \U \}$ with $(\xi',{\xi_n}) \in S^*_{(x',x_n)}M$.  In these coordinates, $\xi'$ is cotangent to $H$ while ${\xi_n}$ is conormal to $H$. 

Note that in the Fermi coordinate system we have 
\begin{equation}\label{E: R}
|(\xi', {\xi_n})|^2_{g(x', x_n)}= \xi_n^2 + R(x',x_n,\xi'),
\end{equation}
where $R$ satisfies that $R(x',0, \xi')= |\xi'|^2_{g_H(x')}$ for all $(x', \xi') \in T^*H$ and $g_H$ is the Riemannian metric induced on $H$ by $g$.

\subsection{Transversals for defect measures}
We now apply Lemma~\ref{l:invariance} to the special case of defect measures, using the fact that they are invariant under the geodesic flow.  In what follows we write $|\xi'|_{x'}:= |\xi'|_{g_H(x')}$, where $g_H$ is the Riemannian metric  on $H$ induced by $g.$
Let
$$\G_H(\delta):=\{(x,\xi)\in S^*_HM:\; |\xi'|_{x'}^2\geq 1-\delta^2\},$$
and define the set of non-glancing directions
$$\Sigma_\delta:= S^*_HM\setminus \G_H(\delta).$$ 
\begin{lemma}\label{l:defectInvariance1}
Suppose $\mu$ is a defect measure associated to a sequence of Laplace eigenfunctions. Then, for all $\delta>0$ there exists $\e>0$ small enough so that
\begin{gather*} 
\iota^*\mu =dt\, d\mu_{\Sigma_\delta}\quad \quad \text{on }\;(-\e,\e)\times \Sigma_\delta
\end{gather*}
where 
$$\iota:(-\e,\e)\times \Sigma_\delta \to  \bigcup_{|s|<\e}G^s(\Sigma_\delta), \qquad \qquad \iota(t,q)= G^t(q),$$
is a diffeomorphism and $d\mu_{\Sigma_\delta}$ is a finite Borel measure on $\Sigma_\delta$.
\end{lemma}
\begin{proof}  In what follows we use Lemma \ref{l:invariance} with $N=S^*M$, $\mathcal V=H_p$ the Hamiltonian flow for $p=|\xi|_g$, and $\varphi_t^{\mathcal V}=G^t$ the geodesic flow.
Note that since $\mu$ is a defect measure for a sequence of Laplace eigenfunctions, it is invariant under the geodesic flow $G^t$.
Then, for  $q\in \Sigma_\delta$,
$$|H_px_n(q)|  >c\, \delta>0$$
 and hence $\Sigma_\delta$ is transverse to $G^t$. Therefore, there exists $\e>0$ so that 
$\iota:(-2\e,2\e)\times \Sigma_\delta \to S^*M,$ with $\iota(t,q)=G^t(q),$
is a coordinate map.
\end{proof}

\begin{remark}\label{l:pushforward}
For each $A\subset S^*_HM$ with $\overline{A}\subset S^*_HM\setminus S^*H$, there exists $\delta_0>0$ so that 
$$
 d\mu_{\Sigma_\delta}(A)=\lim_{t\to 0}\frac{1}{2t}\mu\Big(\bigcup_{|s|\leq t} G^s(A)\Big)
$$
for all $0<\delta \leq \delta_0$. Indeed, since $\overline A$ is compact, there exists $\delta_0=\delta_0(A)>0$ so that $\overline A \subset \Sigma_{\delta_0}$. Then, by Lemma \ref{l:defectInvariance1}, there exists $\ep=\ep(A)>0$ so that if $|t| \leq \ep$, then 
$$
 \mu\Big(\bigcup_{|s|\leq t} G^s(A)\Big)=  2t  \, d\mu_{\Sigma_\delta}(A). 
$$
In particular, we conclude that the quotient $\frac{1}{2t}\mu\Big(\bigcup_{|s|\leq t} G^s(A)\Big)$ is independent of $t$ as long as $|t|\leq \ep$.
\end{remark}

We also need the following description of $\mu$.
\begin{lemma}
\label{l:defectInvariance} 
Suppose $\mu$ is a defect measure associated to a sequence of Laplace eigenfunctions, and let $\delta>0$. Then, in the notation of Lemma \ref{l:defectInvariance1}, there exists $\e_0>0$ small enough so that 
$$
\mu =|\xi_n|^{-1} d\mu_{\Sigma_{\delta}}(x',\xi',\xi_n) dx_n,
$$
 for $(x',x_n,\xi',\xi_n)\in \iota((-\e_0,\e_0)\times \Sigma_\delta)$.
\end{lemma}
\begin{remark}
\label{remGeom}
Notice that $|H_px_n|>\gamma $ on $\Sigma_\gamma=S^*_HM\setminus\mathcal{G}(\gamma)$. Therefore, there exists $c_0,c_1>0$ so that 
$$
 \{(x',x_n, \xi', \xi_n):\;\; |x_n|\leq c_0\gamma,\; |\xi'|_{x'}^2\leq 1-c_0^{-1}\gamma^2\}\subset \bigcup_{|t|\leq c_1\gamma}G^t(\Sigma_\gamma).
 $$
\end{remark}
\begin{proof}
By Lemma~\ref{l:defectInvariance1},  
$$\iota^*\mu =d\mu_{\Sigma_{\delta}}(x',\xi',\xi_n) dt\quad\text{on }(-\e,\e)\times \Sigma_{\delta}.$$
Then, for  $q\in \Sigma_{\delta}$
$$|\partial_t x_n(\iota(0,q))|=|H_px_n(\iota(0,q))|=\frac{|\xi_n(\iota(0,q))|}{|\xi|_g}> \delta$$
and hence for $\e_0>0$ small enough and $q\in \Sigma_{\delta}$, $t\in (-\e_0,\e_0)$,
$$|\partial_t x_n(\iota(t,q))|=|H_px_n(\iota(t,q))|=\frac{|\xi_n(\iota(t,q))|}{|\xi|_g}> \frac{\delta}{2}.$$
Therefore, $dt= f(x',x_n,\xi',\xi_n)dx_n$ where 
$$f(x',x_n,\xi',\xi_n)=|H_px_n(\iota^{-1}(x_n,(x',\xi',\xi_n)))|^{-1}=\frac{|\xi|_g}{|\xi_n|}=|\xi_n|^{-1}$$ 
where in the last equality, we use that $|\xi|_g=1$. 
In particular, 
$$\mu=|\xi_n|^{-1}d\mu_{\Sigma_{\delta}}(x',\xi',\xi_n)dx_n.$$
\end{proof}


Before proceeding to the proof of Theorem \ref{T: main} we note that Lemma~\ref{l:defectInvariance} implies that for all $\delta>0$, 
\begin{equation}
\label{e:0H}
\mu( S^*_H M\setminus \G_H(\delta))=0.\qquad 
\end{equation}

\begin{remark}
Notice that the measure 
$$|\xi_n|^{-1}d\mu_{\Sigma_\delta}(x',\xi',\xi_n)=\frac{1}{\sqrt{1-|\xi'|_{x'}^2}}d\mu_{\Sigma_{\delta}}(x',\xi',\xi_n)$$ 
is hypersurface measure on $S^*_HM\setminus \mathcal{G}(\delta)$ induced by $\mu$ {where we take $\partial_{x_n}$ to be the normal vector field to $S^*_HM$}. For example, if $\mu_L$ is Liouville measure, then, parametrizing $S^*_HM\setminus \mathcal{G}(\delta)$ by $(x',\xi')$
$$d(\mu_L)_{\Sigma_\delta}=c \mathbf{1}_{\{S^*_HM\setminus \mathcal{G}(\delta)\}}(x',\xi',\xi_n) dx'd\xi'$$
for some $c>0$.
\end{remark}

\section{Proof of Theorem \ref{T: main}}
Consider the cut-off function $\chi_\alpha \in C^\infty(\R, [0,1])$ with 
$$\chi_\alpha(t) =\begin{cases} 
0 & |t| \geq \alpha \\
1 & |t| \leq \frac{\alpha}{2},
\end{cases}$$
with  $|\chi_\alpha'(t)| \leq 3/\alpha$ for all $t \in \R$.

 For $\delta>0$ consider the symbol
 \begin{equation}\label{E: beta}
 \beta_\delta(x',\xi')=\chi_\delta(|\xi'|_{x'}) \in S^0(T^*H)
 \end{equation}
 where we continue to write $|\xi'|_{x'}:= |\xi'|_{g_H(x')}.$  We refer the reader to the Appendix where the semiclassical notation used in this section is introduced.
 The operator $Op_h(\beta_\delta) \in \Psi^0(H)$ microlocalizes near the conormal direction in $T^*H$ which is identified with $\xi'=0$ via the orthogonal projection.  The first step towards the proof of Theorem \ref{T: main} is to reduce the problem to study averages over $H$ of the functions $\phi_h$ and $h\partial_{\nu}\phi_h$ when microlocalized near the conormal direction.
   
\begin{lemma}\label{L: microlocalize}  For any $\delta >0$ and  $ u \in L^2(H),$
\[\int_H u\, d\sigma_H=\int_H Op_h(\beta_\delta) u\, d\sigma_H+O_\delta(h^\infty){\|u\|_{L^2(H)}}.\]
\end{lemma}
\begin{proof}
We wish to show that 
$$\langle (1-Op_h(\beta_\delta)) u, 1 \rangle_{L^2(H)}=\langle  u,  (1-Op_h(\beta_\delta))^* 1 \rangle_{L^2(H)}=O_\delta(h^\infty).$$
 To prove this,  we simply note that {in local coordinates}
\[(1-Op_h(\beta_\delta) )^* 1 (x)= \frac{1}{(2\pi h)^{n-1}}  \iint e^{ \frac{i}{h}\langle x-x',\xi' \rangle} {a_\delta(x,\xi';h)}(1-\chi_{{2}\delta})(|\xi'|_{x}) \, d\xi' dx', \]
{for some symbol $a_\delta\in S^0$}.
{T}he phase function $\Phi(x', \xi';x) =\langle x-x',\xi' \rangle$ has critical points  in $(x', \xi')$ given by 
$$(x', \xi')=(x,0).$$

 By repeated integration by parts with respect to the operator  
\[L:=  \frac{1}{|x-x'|^2 + |\xi'|^2} \left( \sum_{j=1}^n \xi_j'  hD_{x'_j} +  \sum_{j=1}^n (x'_j-x_j) hD_{\xi'_j} \right),\] using that  $L (e^{i\Phi/h}) = e^{i \Phi/h},$ one gets
\[(1-Op_h(\beta_\delta) )^* 1 (x)= \frac{1}{(2\pi h)^{n-1}}  \iint e^{i \frac{(x-x')\xi'}{h}}{a_\delta(x,\xi')} (1-\chi_\delta)(|\xi'|_{x}) \chi_1(|x-x'|) \,d\xi' dx' + O_\delta(h^\infty) \]
$$ = O_{\delta}(h^{\infty}),$$
uniformly in $x \in H.$ The last line follows by repeated integrations by parts with respect to $L $ using the fact that $(1-\chi_\delta)^{(k)}(0)=0$ for all $k  \geq 0.$ 
\end{proof}

\subsection{Proof of Theorem \ref{T: main}}
We wish to show that for any $\ep>0$ there exists $h_0(\ep)>0$ so that 
\begin{equation}\label{E: to prove 0}
 \left|\int_H \phi_h \, d\sigma_H\right| \leq  \ep \qquad \text{and}\qquad \left| \int_H  h \partial_\nu \phi_h d\sigma_H \right| \leq \ep, 
\end{equation}
for all  $h \leq h_0$. 

In view of  Lemma \ref{L: microlocalize}, we can microlocalize the problem to the conormal direction; that is, the claim in \eqref{E: to prove 0}  follows provided we prove that given $\ep>0$ there exist $\delta(\ep)>0$ and $h_0(\ep)>0$ so that 

\begin{equation}\label{E: to prove 1}
 \left|\int_H  Op_h(\beta_{\delta}) \phi_h \, d\sigma_H\right| \leq  \ep \qquad \text{and}\qquad \left| \int_H  Op_{h}(\beta_{\delta}) h \partial_\nu \phi_h d\sigma_H \right| \leq \ep, 
\end{equation}
for all  $h \leq h_0(\ep)$.

To prove (\ref{E: to prove 1}), by Cauchy-Schwarz, it clearly suffices to establish the stronger bounds

\begin{equation}\label{E: to prove}
  \|Op_h(\beta_\delta) \phi_h\|_{L^2(H)} \leq \ep \qquad \text{and}\qquad \|Op_h(\beta_\delta)  h \partial_\nu \phi_h\|_{L^2(H)} \leq  \ep, 
\end{equation}\\
for all  $h \leq h_0(\ep)$ and $\delta(\ep) >0$ sufficiently small.

From now on, we fix $\ep>0$.  Using Green's formula \cite{CTZ}, it is straightforward to check that for any operator $A:C^\infty(M)\to C^\infty(M)$  one has the Rellich Identity
\begin{equation}\label{E: Green}
\frac{i}{h} \int_{\Omega_H} [-h^2 \Delta_g, A] \phi_h \, \overline{\phi_h} \, dv_g= \int_H A \phi_h\, \overline{hD_\nu \phi_h} \, d\sigma_H + \int_H h D_\nu (A \phi_h) \, \overline{\phi_h} \, d\sigma_H,
\end{equation}
where $D_\nu = \frac{1}{i} \partial_\nu$, with $\nu$ being the unit outward vector normal to $\Omega_H$.

Let $\delta>0$ and $\alpha>0$ be two real valued parameters to be specified later and consider the operator
\[A_{\delta, \alpha}(h):=Op_h(\beta_\delta^2) \circ Op_h(\chi_\alpha(x_n)) \circ hD_\nu,\]
where $\beta_\delta$ is defined in \eqref{E: beta}.
{\begin{remark}
We note that when we write $Op_h (\beta_\delta^2)$ above, we are actually considering the operator $Op_h(\beta_\delta^2)\otimes \Id_{x_n}$. That is, for $u\in C^\infty(M)$,
$$[Op_h(\beta_\delta^2)u](x',y_n)=[Op_h(\beta_\delta^2)u|_{x_n=y_n}](x').$$
\end{remark}
}
 The operator $A_{\delta, \alpha}(h)$ is the semiclassical normal derivative operator  $h$-microlocalized to a neighbourhood of the conormal direction to $H$ over the collar neighbourhood $U_H.$

We note that 
 \begin{equation}\label{E: RHS 1}
  \int_H A_{\delta, \alpha}(h) \phi_h\, \overline{hD_\nu \phi_h} \, d\sigma_H= \left \langle Op_h(\beta_\delta^2)  hD_\nu \phi_h\,,\, {hD_\nu \phi_h} \right\rangle_{L^2(H)},
  \end{equation}
  since $\chi_\alpha(x_n)=1$ for $x_n \in [-\frac{\alpha}{2},\frac{\alpha}{2}]$.
 Without loss of generality, we may assume that 
\[\Omega_H \cap U_H =\{(x', x_n): \, x'\in H \;\text{and}\; x_n<0\}.\]
With this choice, $D_\nu=D_{x_n}$.
We next recall that  
 $$ \gamma_H (h^2 D_{\nu}^2 \phi_h) = (I + h^2 \Delta_{g_H}) \gamma_H( \phi_h) + h a_1 \gamma_H ( \phi_h) + h a_2 \gamma_H  (h D_{\nu} \phi_h) ,$$ 
 where $\gamma_H:M \to H$ is the restriction map to $H$, and $a_1, a_2 \in C^{\infty}(H)$. 
 Since $\chi_\alpha'(0)=0$ it follows from the restriction upper bounds $\|\phi_h\|_H=O(h^{-1/4})$ \cite{BGT,HTacy,T,Tat} and $\| h D_\nu \phi_h \|_H = O(1)$ \cite{CHT,T14} 
that
$$ \langle (hD_\nu)^2 \phi_h,\phi_h \rangle_{L^2(H)} - \langle (1+h^2\Delta_{g_H})\phi_h, \phi_h \rangle_{L^2(H)}  = O_{L^2}(\sqrt{h}).$$

Consequently,
 \begin{align}
  \int_H h D_\nu (A_{\delta, \alpha}(h) \phi_h) \, \overline{\phi_h} \, d\sigma_H
  &=\left \langle h D_\nu  Op_h(\beta_\delta^2)\chi_{\alpha}(x_n)  h D_\nu  \phi_h \,,\, { \phi_h} \right\rangle_{L^2(H)}  \notag\\
    &=\left \langle  Op_h(\beta_\delta^2) ( h D_\nu)^2  \phi_h \,,\, { \phi_h} \right\rangle_{L^2(H)}  \notag\\
  &= \left \langle Op_h(\beta_\delta^2)  (1+h^2\Delta_{g_H}) \phi_h \,,\, { \phi_h} \right\rangle_{L^2(H)} + O(h^{\frac{1}{2}}).\label{E: RHS 2}
  \end{align}

Substitution of  \eqref{E: RHS 1} and \eqref{E: RHS 2} in \eqref{E: Green} gives
\begin{align}\label{E: Green 2}
&\frac{i}{h} \int_{\Omega_H} [-h^2 \Delta_g, A_{\delta, \alpha}(h)] \phi_h \, \overline{\phi_h} \, dv_g= \notag \\
&\qquad = \left \langle Op_h(\beta_\delta^2)  hD_\nu \phi_h\,,\, {hD_\nu \phi_h} \right\rangle_{H}+  \left \langle Op_h(\beta_\delta^2)  (1+h^2\Delta_{g_H}) \phi_h \,,\, { \phi_h} \right\rangle_{H} + O(h^{\frac{1}{2}}).
\end{align}

Next, we observe that 
\begin{equation}\label{E: E1}
\|Op_h(\beta_\delta)  hD_\nu \phi_h\|^2_H= \left \langle Op_h(\beta_\delta^2)  hD_\nu \phi_h\,,\, {hD_\nu \phi_h} \right\rangle_{H}+O(h)
\end{equation}
 since $\|hD_\nu \phi_h\|_H=O(1)$  \cite{CHT}.
On the other hand, for $(x',\xi') \in {\supp} \, \beta_\delta$ we have $|\xi'|^2_{x} \leq \delta^2$ and so,
\[\beta_\delta^2 \cdot (1- |\xi'|^2_{x}) - \beta_\delta^2 \cdot (1-2\delta^2)= \beta_\delta^2 \left(2 \delta^2 - |\xi'|^2_{x}\right)  \geq \beta_\delta^2 \delta^2\geq 0.\]
 Therefore, combining the sharp Garding inequality with the bound  $\|\phi_h\|_H=O(h^{-1/4})$ gives
\begin{align}
(1-2\delta^2) \|Op_h(\beta_\delta)  \phi_h\|_{H}^2
&= \langle  Op_h(\beta_\delta^2(1-2\delta^2) ) \phi_h, \phi_h \rangle_H + O(h^{\frac{1}{2}}) \notag\\
&\leq  \langle  Op_h(\beta_\delta^2 \cdot (1-|\xi'|^2_{x}) ) \phi_h, \phi_h \rangle_H + O(h^{\frac{1}{2}}) \notag\\
&=  \langle  Op_h(\beta_\delta^2) (1+h^2\Delta_{g_H} ) \phi_h, \phi_h \rangle_H + O(h^{\frac{1}{2}}). \label{E: E2}
\end{align} 
Substitution of  \eqref{E: E1} and \eqref{E: E2} into  \eqref{E: Green 2} gives

\begin{multline}\label{E: Green 3}
\|Op_h(\beta_\delta)  hD_\nu \phi_h\|_H^2 +(1-2\delta^2) \|Op_h(\beta_\delta)  \phi_h\|_{H}^2  \leq \\\frac{i}{h} \int_{\Omega_H} [-h^2 \Delta_g, A_{\delta, \alpha}(h)] \phi_h \, \overline{\phi_h} \, dv_g + O(h^{\frac{1}{2}}).
\end{multline}

The claim in \eqref{E: to prove} follows at  once from \eqref{E: Green 3} provided we show that for any $\ep >0$ there exist $\delta, \alpha>0$ and $h_0>0$  (all possibly depending on $\ep$) such that 
\begin{equation} \label{E: goal}
\Big|\Big \langle \frac{i}{h} [-h^2 \Delta_g, A_{\delta, \alpha}(h)] \phi_h \,,\, {\phi_h} \Big \rangle_{L^2(\Omega_H)} \Big| \leq \ep^2 \qquad \forall h \leq h_0(\ep)\end{equation}

To prove \eqref{E: goal} we note that 
\begin{multline} \label{E: goal 2}
\Big \langle \frac{i}{h} [-h^2 \Delta_g, A_{\delta, \alpha}(h)] \phi_h \,,\, {\phi_h} \Big \rangle_{L^2(\Omega_H)}=
\\=\Big \langle Op_h\Big( \big\{\sigma(-h^2\Delta_g), \sigma(A_{\delta, \alpha}(h))\big\} \Big) \phi_h \,,\, {\phi_h} \Big \rangle_{L^2(\Omega_H)}+O(h),
\end{multline}
where $\sigma(A_{\delta, \alpha}(h))(x,\xi)= \beta_\delta^2(x',\xi') \,\chi_\alpha(x_n) \xi_n $, and  according to \eqref{E: R},  the Poisson bracket 
\begin{equation}\label{E: poisson}
 \big\{|(\xi', {\xi_n})|^2_{x}\,,\, \sigma(A_{\delta, \alpha}(h))\big\} = 2\chi_\alpha'(x_n) \beta_\delta^2(x',\xi') \, \xi_n^2 +  \chi_\alpha(x_n)q_\delta(x',x_n,\xi', {\xi_n}) 
 \end{equation}
where,
 $$q_\delta(x,\xi):=  \xi_n\partial_{\xi'} R \cdot \partial_{x'} \beta_\delta^2- \xi_n\partial_{x'} R \cdot\partial_{\xi'} \beta_\delta^2    -\partial_{x_n} R \cdot \beta_\delta^2 .$$\\
  We now estimate each term in the RHS of \eqref{E: poisson} separately.

\begin{lemma}\label{L: lemma}
Let $\{ \phi_h \}$ be an $L^2$-normalized eigenfunction sequence with defect measure $\mu.$ Then,  \\

\noindent (i)\;\; $  | \big\langle Op_h( \chi_\alpha(x_n) q_\delta ) \phi_h \,,\, {\phi_h} \big \rangle_{L^2(\Omega_H)} |  \leq R_{\alpha,\delta}+o(1),$\\
where 
\[R_{\alpha,\delta} := \| q_{\delta} \|_{L^{\infty}} \cdot \mu(\{(x',x_n,\xi', {\xi_n}) \in S_{U_H}^*M:\, |x_n|\leq \alpha, |\xi'| < \delta \})^{\, \frac{1}{2}}. \]
In addition,
\begin{equation*}\label{E: claim 1}
(ii) \;\; \big\langle Op_h( 2\chi_\alpha'(x_n) \beta_\delta^2(x',\xi') \, \xi_n^2  ) \phi_h \,,\, {\phi_h} \big \rangle_{L^2(\Omega_H)}=\int_{S^*_{\Omega_H} M} 2\chi_\alpha'(x_n) \beta_\delta^2(x',\xi') \, \xi_n^2\, d\mu  +o(1).
\end{equation*} 
In both (i) and (ii),  $o(1)$ denotes a term that vanishes as $h \to 0^+.$
\end{lemma}

We postpone the proof of Lemma \ref{L: lemma} until the end of this section. Assuming this result for the moment, we now conclude the proof of the theorem. 
From  Lemma \ref{L: lemma}  and \eqref{E: poisson}, it follows that
\begin{align} \label{E: goal 4}
\Big \langle \frac{i}{h} [-h^2 \Delta_g, A_{\delta, \alpha}(h)] \phi_h \,,\, {\phi_h} \Big \rangle_{L^2(\Omega_H)}
=\int_{S^*_{\Omega_H}M} 2\chi_\alpha'(x_n) \beta_\delta^2(x',\xi') \, {\xi_n}^2\, d\mu+  R_{\alpha,\delta} + o(1).
\end{align}
Since $\mu$ is a Radon measure, and hence monotone,
 \begin{equation} \label{ir}
\lim_{\alpha \to 0} R_{\alpha,\delta} =  \| q_{\delta} \|_{L^{\infty}} \cdot \mu( \{ (x',0,\xi) \in S_H^*M; \, |\xi'| < \delta \} )^{\frac{1}{2}}. \end{equation}
Thus, using Lemma~\ref{l:defectInvariance} (or more precisely~\eqref{e:0H}) gives
 \begin{equation} \label{ir2}
\lim_{\alpha \to 0} R_{\alpha,\delta} =0.
 \end{equation}

Moreover, since the LHS of (\ref{E: Green 3}) is independent of $\alpha$, we are free to take the $\alpha \to 0$ limit of both sides. In view of (\ref{E: goal 4}) and (\ref{ir}), it follows that after taking $h \to 0^+$ and then $\alpha \to 0^+,$

\begin{align}\label{E: Green 4}
&\limsup_{h \to 0} \Big( \|Op_h(\beta_\delta)  hD_\nu \phi_h\|_H^2 +(1-2\delta^2) \|Op_h(\beta_\delta)  \phi_h\|_{H}^2 \Big)\leq   \nonumber \\
&\hspace{4cm}\leq   \limsup_{\alpha \to 0} \limsup_{h \to 0^+} \frac{i}{h} \int_{\Omega_H} [-h^2 \Delta_g, A_{\delta, \alpha}(h)] \phi_h \, \overline{\phi_h} \, dv_g \nonumber \\
&\hspace{4cm}= \limsup_{\alpha \to 0^+} \int_{S^*_{\Omega_H}M} 2\chi_\alpha'(x_n) \beta_\delta^2(x',\xi') \, \xi_n^2\, d\mu. \quad 
\end{align}\\
The last line in (\ref{E: Green 4}) follows from (\ref{ir2}). 

To analyze the RHS of \eqref{E: Green 4},  fix  $\gamma>0$ small. By Lemma \ref{l:defectInvariance}  there exists $\ep_\gamma>0$ and a measure $\mu_{\Sigma_\gamma}$ on $\Sigma_\gamma=\{(x,\xi)\in S^*_HM:\; |\xi'|_{x'}^2\leq 1-\gamma^2\}$ so that 
\[
\mu(x,\xi) =f(x',x_n,\xi',\xi_n) d\mu_{\Sigma_{\gamma}}(x',\xi',\xi_n) dx_n, \hspace{1cm}(x,\xi)\in \bigcup_{|t|\leq \ep_\gamma} G^t(\Sigma_\gamma).
\]
By Remark~\ref{remGeom} we may assume that we work with  $\alpha, \delta$  small enough so that 
\begin{equation}
\label{e:supp}
{\supp}(\chi_\alpha' \cdot \beta_\delta^2 ) \subset \bigcup_{|t|\leq \ep_\gamma} G^t(\Sigma_\gamma).
\end{equation}

Since  supp$(\chi_\alpha' ) \subset (-\alpha, 0),$ by the Fubini theorem we have
\begin{equation}\label{E: to bound}
\begin{aligned}
&\int_{S^*_{\Omega_H}M} 2\chi_\alpha'(x_n) \beta_\delta^2(x',\xi')  \xi_n^2\, d\mu=\\
&\qquad =\int_{-c}^0  2\chi_\alpha'(x_n)  \left( \int_{S_{H}^*\!M}\beta_\delta^2(x',\xi')  \xi_n^2|\xi_n|^{-1} d\mu_{\Sigma_{\gamma}} (x', \xi',{\xi_n}) \right) dx_n\\
&\qquad =\int_{S_{H}^*\!M}\int_{-c}^0  2\chi_\alpha'(x_n)  \beta_\delta^2(x',\xi')  |\xi_n|  dx_nd\mu_{\Sigma_{\gamma}} (x', \xi',{\xi_n}).
\end{aligned}
\end{equation}
Sending $\alpha\to 0$ gives 
$$
\lim_{\alpha\to 0^+}\int_{S^*_{\Omega_H}M} 2\chi_\alpha'(x_n) \beta_\delta^2(x',\xi')  \xi_n^2\, d\mu=\int_{S_{H}^*\!M} 2  \beta_\delta^2(x',\xi') |\xi_n| d\mu_{\Sigma_{\gamma}} (x', \xi',{\xi_n}).
$$
Sending $\delta\to 0$ and using that $\beta_\delta\equiv 1$ on $N^*H$, $|\beta_\delta|\leq C$ we obtain
\begin{multline}
\label{E: bound on int}
\lim_{\delta\to 0}\lim_{\alpha\to 0^+}\int_{S^*_{\Omega_H}M} 2\chi_\alpha'(x_n) \beta_\delta^2(x',\xi')  \xi_n^2\, d\mu=\int_{N^*H} 2 \, d\mu_{\Sigma_{\gamma}}(x', \xi',{\xi_n})=2\, \mu_{\Sigma_{\gamma}}(N^*H).
\end{multline}
Since $\mu$ is conormally diffuse, we have by Remark~\ref{l:pushforward} that $\mu_{\Sigma_{\gamma}}(N^*H) =0$ and so (\ref{E: to prove}) follows from \eqref{E: bound on int} and \eqref{E: Green 4}. \qed

\subsection{Proof of Lemma \ref{L: lemma}}

\begin{proof}
First, we use the standard fact that $\{\phi_h\}$ are microsupported on $S^*M$ \cite{CHT} to $h$-microlocally cut them off near $S^*M$. 
More precisely, for  $r>0$ small, consider the annular shell $$A(r):= \left\{(x,\xi) \in T^*M: \, 1-r < |\xi|_{g(x)}< 1+r \right\}.$$  Let  $\tilde \chi \in C_c^\infty(T^*M)$ be a cutoff function equal to  $1$ on $A(r)$ and zero on $T^*M \setminus A(2r)$. Then, \cite{CHT}
\begin{equation}\label{E: chi tilde}
\| \phi_h-Op_h(\tilde \chi) \phi_h \|_{L^2(M)}=O(h^\infty).
\end{equation}

\noindent \emph{Proof of (i):} 
Since $\|\phi_h\|_{L^2(M)}=1$, by Cauchy-Schwarz,
\begin{align*}
\big|\big\langle Op_h( \chi_\alpha(x_n) q_\delta ) &\phi_h \,,\, {\phi_h} \big \rangle_{L^2(\Omega_H)}\big|^2 \leq  \| Op_h( \chi_\alpha(x_n) q_\delta ) \phi_h \|^2_{L^2(M)}\\
&=\big \langle [ Op_h( \chi_\alpha(x_n) q_\delta )]^* [ Op_h( \chi_\alpha(x_n) q_\delta )] \phi_h, \phi_h \big \rangle_{L^2(M)} \\
&=\big\langle [ Op_h( \chi_\alpha(x_n) q_\delta )]^* [ Op_h( \chi_\alpha(x_n) q_\delta )] \phi_h,  Op_h(\tilde \chi) \phi_h \big\rangle_{L^2(M)}  + O(h^\infty)\\
&=\big\langle Op_h( \tilde \chi \cdot \chi_\alpha^2(x_n) \cdot |q_\delta|^2 ) \phi_h, \phi_h \rangle_{L^2(M)}  + O(h)\\
&=\int_{S^*  M} \tilde \chi \cdot \chi_\alpha^2(x_n) \cdot |q_\delta|^2 \, d\mu    + o(1)\\
&\leq \| q_{\delta} \|^2_{L^\infty}  \cdot \mu \big(\{(x',x_n,\xi', {\xi_n}) \in S_{U_H}^*M:\, |x_n|\leq \alpha, |\xi'| < \delta \} \big) +o(1),
\end{align*}
where the penultimate identity follows from the fact that $\mu$ is the defect measure associated to $\{\phi_h\}$ and  the symbol $  \tilde \chi \cdot \chi_\alpha^2(x_n) \cdot |q_\delta|^2  \in C_c^\infty (T^* U_H)$. \ \\

\noindent \emph{Proof of (ii):}
Let $\rho \in C_c^\infty(\R)$ be a smooth cut-off function with $\rho(x_n)=0$ for $x_n\geq 0$ and $\rho(x_n)=1$ for $x_n \leq -\alpha/2$. Then,  since $\Omega_H \cap U_H$ is identified with the set of points on which $x_n<0$, and ${\supp} (\chi_\alpha')  \subset (-\infty, -\alpha/2] \cup [\alpha/2, +\infty)$, we have
$$\rho(x_n) \chi_\alpha'(x_n)= \begin{cases} 0 & \text{on}\; \Omega_H^c,\\ \chi_\alpha'(x_n) & \text{on}\; \Omega_H. \end{cases}$$

Note that since $\chi_\alpha'(x_n)=0$ for $x_n \in [-\alpha/2, \alpha/2]$,  we may regard $\rho\chi_\alpha'$ as a smooth function defined on all of $M$. We then have that
\begin{multline*} 
\big\langle Op_h( 2\chi_\alpha'(x_n) \beta_\delta^2(x',\xi') \, \xi_n^2  ) \phi_h \,,\, {\phi_h} \big \rangle_{L^2(\Omega_H)}=\\
=\big\langle Op_h( 2 \rho(x_n) \chi_\alpha'(x_n) \beta_\delta^2(x',\xi') \, \xi_n^2  ) \phi_h \,,\, {\phi_h} \big \rangle_{L^2(M)}.
\end{multline*}
Microlocalizing the eigenfunctions near $S^*M$ by using the cut-off $\tilde \chi $ we obtain
\begin{multline*} 
\big\langle Op_h(  2 \, \chi_\alpha'(x_n) \beta_\delta^2(x',\xi') \, \xi_n^2  ) \phi_h \,,\, {\phi_h} \big \rangle_{L^2(\Omega_H)}=\\
=\big\langle Op_h( \tilde \chi \rho(x_n) 2 \,\chi_\alpha'(x_n) \beta_\delta^2(x',\xi') \, \xi_n^2  ) \phi_h \,,\, {\phi_h} \big \rangle_{L^2(M)} + O(h).\end{multline*}
Using that $\mu$ is the defect measure associated to $\{\phi_h\}$, and that the symbol $  \tilde \chi   \beta_\delta^2 \, \xi_n^2   \in C_c^\infty(T^*M)$, we obtain
\begin{align*}
&\big\langle Op_h( \tilde \chi \rho(x_n) 2\chi_\alpha'(x_n) \beta_\delta^2(x',\xi') \, \xi_n^2  ) \phi_h \,,\, {\phi_h} \big \rangle_{L^2(M)}= \\
&\hspace{6cm}=\int_{S^*M} 2\rho(x_n)\chi_\alpha'(x_n) \beta_\delta^2(x',\xi') \, \xi_n^2\, d\mu+o(1) \\
&\hspace{6cm}=\int_{S^*_{\Omega_H}M} 2\chi_\alpha'(x_n) \beta_\delta^2(x',\xi') \, \xi_n^2\, d\mu+o(1),
\end{align*}
as claimed.
\end{proof}

\begin{remark} \label{localize} By replacing the test operator $A_{\delta,\alpha}(h)$ with

\[\tilde{A}_{\delta, \alpha}(h):=Op_h(\beta_\delta^2(x',\xi')) \circ  f(x') \circ Op_h(\chi_\alpha(x_n)) \circ hD_\nu,\]

\noindent where $f \in C^{\infty}(H)$ and carrying out the same argument as in the proof of Theorem \ref{T: main}, it is easy to see that under the assumption $\mu_H( \pi^{-1}(\supp f)\cap N^*H)=0$,
$$ \int_H f \, \phi_h d\sigma_H = o(1) \quad \text{and} \quad \int_H f \,  h D_{\nu} \phi_h d\sigma_H = o(1).$$
\end{remark}

\section{Proof of Theorem \ref{local}}
To prove Theorem \ref{local} we need the following result.
\begin{lemma}
\label{l:rough}
Suppose $A\subset H$ has piecewise smooth boundary. Then for all $\epsilon>0$
$$
\|(1-Op_h(\beta_\delta))^* 1_A\|_{L^2(H)}=O_{\epsilon}(h^{\frac{1}{2}-\epsilon}).
$$
\end{lemma}
\begin{proof}
To prove this result we first introduce a cut-off function $\chi_h $ so that  $(1-\chi_h)1_A$  is smooth and close to $1_A$.
Let $\chi_h \in C_c^\infty(H)$  satisfy
\begin{enumerate}
\item[i)] $\chi_h \equiv 1\quad \text{ on }\{x\in H:\; d(x,\partial A)\leq h^{1-\epsilon}\}$. \medskip
\item[ii)]  $\supp \chi \subset \{x\in H:\; d(x,\partial A)\leq 2h^{1-\e}\}$.  \medskip
\item[iii)]  $|\partial_x^\alpha \chi |\leq C_\alpha h^{|\alpha|(1-\e)}.$ \label{e:derEst} \medskip
\end{enumerate}
Then, $(1-\chi_h)1_A$ satisfies the same bound as in (iii), and hence integrating by parts as in Lemma~\ref{L: microlocalize} i.e. with 
$L:= \frac{1}{|x-x'|^2 + |\xi'|^2} \left( \sum_{j=1}^n \xi_j'  hD_{x'_j} +  \sum_{j=1}^n (x'_j-x_j) hD_{\xi'_j} \right),$
 gives
\begin{align*}
&[(1-Op_h(\beta_\delta))^*(1-\chi_h)1_A](x)\\
&\qquad=\frac{1}{(2\pi h)^{n-1}}\iint e^{\frac{i}{h}\langle x-x',\xi'\rangle} (1-\beta_\delta(x',\xi'))(1-\chi_h(x'))1_A(x')dx'd\xi'\\
&\qquad= \frac{1}{(2\pi h)^{n-1}}\iint e^{\frac{i}{h}\langle x-x',\xi'\rangle}(L^*)^N\big[(1-\beta_\delta(x',\xi'))(1-\chi_h(x'))1_A(x')\big]dx'd\xi'\\
&\qquad=O_N(h^{1-n+N(1-\e)}).
\end{align*}
In particular,
\begin{equation}
\label{e:awayEst}
\|(1-Op_h(\beta_\delta))^*(1-\chi_h)1_A\|_{L^\infty}=O_{\epsilon}(h^\infty).
\end{equation}
On the other hand
\begin{equation}
\label{e:nearEst}
\|\chi_h1_A\|_{L^2(H)}=O(h^{\frac{1-\epsilon}{2}}).
\end{equation}
Combining~\eqref{e:awayEst} and~\eqref{e:nearEst} together with $L^2$ boundedness of $Op_h(\beta_\delta)$ proves the lemma.
\end{proof}

\subsection{Proof of Theorem \ref{local}}
 Let $A \subset H$ be an open subset with piecewise $C^{\infty}$ boundary and indicator function $\chi_{A}.$ Suppose that $U\subset H$ is open with $\overline{A}\subset U$. Then since $C^{\infty}(H)$ is dense in $L^2(H)$, for any $\ep>0$, we can find $f \in C^{\infty}(H)$
$$ \| f - 1_{A} \|_{L^2(H)} \leq \ep,\qquad \supp f\subset U.$$
Now,
 \begin{align}
 &\Big| \int_{H} 1_A \phi_h d\sigma_H  \Big| \leq \notag\\
 &  \leq \Big| \int_H 1_AOp_h(\beta_\delta) \phi_h d\sigma_H \Big| + \Big| \langle(1-Op_h(\beta_\delta)) \phi_h ,1_A\rangle_H \Big|\nonumber\\
& \leq \Big| \int_H (1_A- f) Op_h(\beta_\delta) \phi_h d\sigma_H \Big|+ \Big| \int_H f Op_h(\beta_\delta) \phi_h d\sigma_H \Big|+ \Big| \langle \phi_h, (1-Op_h(\beta_\delta))^*1_A \rangle_H \Big|\nonumber  \\
 &\leq\Big| \int_H (1_A- f) Op_h(\beta_\delta) \phi_h d\sigma_H \Big| + o(1). \label{e:octopus}
\end{align}
The last line follows by applying Lemma~\ref{l:rough}, the universal upper bound $\|\phi_h\|_{L^2(H)}\leq Ch^{-\frac{1}{4}}$ \cite{BGT} and Cauchy-Schwarz to the third term, and by applying Remark~\ref{localize} to the second term.

Now, since $\beta_\delta$ is supported away from 
$$S^*H:=\{(x',\xi')\in T^*H:\; |\xi'|_{x'}=1\},$$
we have that $\|Op_h(\beta_\delta)\phi_h\|_{L^2(H)}\leq C$ \cite{BGT,T} and hence applying Cauchy--Schwarz to~\eqref{e:octopus}
$$
\Big|\int_H 1_A\phi_hd\sigma_H\Big|\leq C\e +o(1).
$$
Since $\e>0$ was arbitrary, the theorem follows.
\qed

\begin{remark}
It is clear from the proof of Theorem~\ref{local} that one can decrease the regularity assumption on $\partial A$ and only assume that $\partial A$ has Minkowski box dimension $<n-\frac{3}{2}$ where $n=\dim M$. However, we do not pursue this here.
\end{remark}

%
%
\section{Examples}\label{S: examples}

\subsection{Non vanishing averages on the torus}\label{S:torus}
Let $\mathbb T^2$ be the $2$-dimensional square flat torus. We identify $\mathbb T^2$ with $\{(x_1,x_2): \, (x_1,x_2) \in [0,1)\times [0, 1)\}$. Consider the sequence of normalized eigenfunctions
\[\phi_h(x_1,x_2)= e^{\frac{i}{h} x_1}.\]
Consider the curve $H \subset \mathbb T^2$ defined as $H=\{(x_1,x_2):\;x_1=0\}$. Then, since $\phi_h|_H \equiv 1$, we have 
\[\int_H \phi_h d\sigma_H =1, \quad h^{-1} \in 2\pi \Z^+.\]
We claim that in this case the measure $\mu$ associated to $\{\phi_h\}$ is \emph{not} conormally diffuse with respect to $H$. Actually, we next prove that 
\begin{equation}\label{E: mu on torus}
\mu(x_1, x_2, \xi_1, \xi_2)=\delta_{(1,0)}(\xi_1, \xi_2) \cdot dx_1 \, dx_2, \quad\qquad (x,\xi) \in S^* \mathbb T^2.
\end{equation}
Given \eqref{E: mu on torus}, it follows that 
\[\mu_{H}= \delta_{(1,0)}(\xi_1, \xi_2), \quad\qquad (x,\xi) \in S^* \mathbb T^2.\] In particular, 
\[\mu_{H}(N^*H)=1, \]
so the measure $\mu$ is \emph{not} conormally diffuse with respect  to $H$.

To see that \eqref{E: mu on torus} holds, fix any $a \in C^\infty_c(T^* \mathbb T^2)$. Then, 
\[\langle Op_h(a) \phi_h,  \phi_h\rangle = \frac{1}{(2\pi h)^n} \int_{\mathbb T^2} \int_{\mathbb T^2} \int_{\R^2} a(x,\xi) e^{\frac{i}{h} \psi(x,y,\xi)} d\xi dy dx \]
for the phase function
\[\psi(x,y,\xi):= \langle x-y, \xi \rangle + y_1-x_1.\]
We next do Stationary Phase in $(y, \xi)$. The critical points for the phase are $(y,\xi)=(x,(1,0))$. Also,  
$$\text{Hess}_{(y, \xi)} \psi=\begin{pmatrix} 0 & -1 \\ -1 &0\end{pmatrix}.$$
It follows that 
\[\langle Op_h(a) \phi_h,  \phi_h\rangle =  \int_{\mathbb T^2}  a(x, (1,0))  dx =  \int_{S^* \mathbb T^2} a(x,\xi) \, \delta_{(1,0)}(\xi) dx,\]
as claimed.

\subsection{Defect measures that are not Liouville}
As we already pointed out in the Introduction, the assumptions on $\mu$ for being conormally diffuse are much weaker than asking $\mu$ to be absolutely continuous with respect to the Liouville measure on $S^*M$.
In these examples we build a defect measure $\mu$ that is not absolutely continuous with respect to the Liouville measure but still satisfies the hypothesis of  Theorem \ref{T: main} for a suitable choice of curve $H$.

\subsubsection{Toral Eigenfunctions} Let $\mathbb T^2$ be the $2$-dimensional square flat torus.
We identify $\mathbb T^2$ with  $\{(x_1,x_2): \, (x_1,x_2) \in [0,1)\times [0, 1)\}$. Consider the sequence of eigenfunctions
\[\phi_h(x_1,x_2)=e^{\frac{i}{h}x_1}, \quad h^{-1} \in 2\pi \Z.\] 
As shown in Section \ref{S:torus}, the associated defect measure is $$\mu(x_1, x_2, \xi_1, \xi_2)= \delta_{(1,0)}(\xi_1, \xi_2)dx_1 \, dx_2.$$ 
Next, consider the curve $H \subset \mathbb T^2$ defined as $H=\{(x_1,x_2):\;x_2=0\}$.  
Since 
$N^*H=\{(x_1, x_2, \xi_1, \xi_2) \in S^* \mathbb T^2:\; \xi_1=0\}$, we have for $\delta >0$ sufficiently small,
\[\mu_{H}( N^*H )=0.\] 

Theorem \ref{T: main} therefore implies that
\[\lim_{h\to 0^+}\int_H \varphi_h d\sigma = \lim_{h\to 0^+} \int_0^1 e^{\frac{i x_1}{h}} dx_1 = 0.\]

Of course, in this case the much stronger result $ \int_0^1 e^{\frac{i x_1}{h}} dx_1= 0$ holds for all $h^{-1} \in 2\pi \Z.$ 

\subsubsection{Gaussian Beams}
Consider the two dimensional sphere   $S^2$ equipped with the round metric, and use coordinates
$$ (\theta,\omega)\mapsto (\cos \theta \cos \omega,\sin \theta\cos \omega,\sin \omega)\in S^2,$$
with $[0,2\pi)\times [-\pi/2,\pi/2]$.
For each of the frequencies $h^{-1}=\sqrt{l(l+1)}$ with $\ell \in \mathbb N$ we associate the Gaussian beam
$$
\phi_h(\theta,\omega)= \frac{1}{2^ll!}\Big(\frac{2l+1}{4\pi (2 l)!}\Big)^{\frac{1}{2}}\, e^{-i l \theta} (\cos \omega)^l .
$$
It is normalized so that
$$\|\phi_h\|_{L^2(S^2)}=1,\qquad (-h^2\Delta_{S^2}-1)\phi_h=0.$$

Then, let $\chi \in C_c^\infty(-1,1)$ with $\chi \equiv 1$ on $[-\frac{1}{2},\frac{1}{2}]$ and define
$$
u_h(\theta,\omega)= \frac{1}{2^ll!}\Big(\frac{2l+1}{4\pi (2 l)!}\Big)^{\frac{1}{2}}\, e^{-il\theta}\chi(\omega){ e^{-l\omega^2/2}}.
$$
Observe that 
$$u_h-\phi_h=o_{L^2}(1),$$
so for the purposes of computing the defect measure, we may compute with $u_h$. Using this, by an elementary stationary phase argument, (see e.g. \cite[Section 5.1]{EZB})
the defect measure associated to $\phi_h$ is 
$$
\mu=\frac{1}{2\pi}\delta_{\{\omega=0,\xi=-1,\zeta=0\}}d\theta
$$
where $\xi$ is dual to $\theta$ and $\zeta$ is dual to $\omega$. Let $H=\{(\theta, \omega): \, \omega=0\}$ be the equator. In particular,   $N^*H=\{(\theta, \omega, \xi, \zeta) \in S^*S^2:\; \omega=0,\xi=0,\zeta=\pm 1\}$. Then, 
$$\mu_H(N^*H)=\mu(\{\omega\in(-t_0,t_0),\xi=0,\zeta=\pm1\})=0$$
and Theorem~\ref{T: main} implies
$$
\int_{H}\phi_h (\theta,0)d\theta=o(1).
$$

\section{Appendix on Semiclassical notation}\label{S:appendix}

We next review the notation used for semiclassical operators and symbols and some of the basic properties. First, recall that for a compact manifold $M$ of dimension $n$, we write
$$S^m(T^*M):=\{a(\cdot;h) \in C^\infty(T^*M):\;  |\partial_x^\alpha\partial_\xi^\beta a(x,\xi;h)|\leq C_{\alpha\beta}(1+|\xi|)^{m-|\beta|}\}.$$
We write $\Psi^m(M)$ for the semiclassical pseudodifferential operators of order $m$ on $M$ and 
$$ Op_h:S^m(T^*M)\to \Psi^m(M)$$
for a quantization procedure with $Op_h(1)=\Id+O_{\mathcal{D}'\to C^\infty}(h^\infty)$ and for $u$ supported in a coordinate patch, $\varphi\in C_c^\infty(M)$ with $\varphi\equiv 1$ on $\supp u$ we have
$$Op_h(a)u(x)=\frac{1}{(2\pi h)^n}\iint e^{\frac{i}{h}\langle x-y,\xi\rangle}\varphi(x)a(x,\xi)u(y)d\xi dy +O_{\mathcal{D'}\to C^\infty}(h^\infty)u.$$
 Then there exists a principal symbol map 
 $$
 \sigma:\Psi^m(M)\to S^m(T^*M)/hS^{m-1}(T^*M) 
 $$
 so that 
 \begin{gather*} 
 Op_h\circ \sigma (A)=A+O_{\Psi^{m-1}}(h),\quad A\in \Psi^m,\qquad\qquad\sigma \circ Op_h=\pi : S^m\to S^m/hS^{m-1},
 \end{gather*}
 where $\pi$ is the natural projection map. Moreover, for $A\in \Psi^{m_1},\,B\in\Psi^{m_2}$, \medskip
\begin{itemize} 
\item $\sigma(AB)=\sigma(A)\sigma(B)\;\; \in S^{m_1+m_2}/hS^{m_1+m_2-1},$ \medskip
\item  $\sigma([A,B])=\frac{h}{i}\big\{\sigma(A),\sigma(B)\big\}\;\;\in hS^{m_1+m_2-1}/h^2S^{m_1+m_2-2},$ \medskip
 \end{itemize}
 where $\{\cdot,\cdot\}$ denotes the poisson bracket. For more details on the semiclassical calculus see e.g. \cite[Chapters 4,14]{EZB} \cite[Appendix E]{ZwScat}. 
 
 Finally, we recall the for any $\{u(h)\}_{0<h<h_0}\subset L^2(M)$ a bounded family of functions, we may extract a subsequence $h_k\to 0$ so that for $a\in C_c^\infty(T^*M)$,
 $$
 \langle Op_h(a)u_{h_k},u_{h_k}\rangle_{L^2(M)}\underset{h_k\to 0}{\to }\int a(x,\xi)d\mu
 $$ 
 for a positive Radon measure $\mu$. We call $\mu$ a \emph{defect measure for $u_{h_k}$}. For $p\in S^m(T^*M)$ real valued, if $u(h)$ solves 
 $$Op_h(p)u=o(h),\qquad \|u(h)\|_{L^2}=1,$$
then for any defect measure $\mu$ associated to $u(h)$, 
$$\supp \mu \subset \{p(x,\xi)=0\},\qquad \exp(tH_p)_*\mu =\mu$$
where $H_p$ denotes the Hamiltonian vector field associated to $p$. See e.g. \cite[Chapter 5]{EZB} for more details.

\bibliography{biblio.bib}
\bibliographystyle{amsalpha}
\end{document}